\documentclass[11pt, reqno]{amsart}
\usepackage{amsmath}
\usepackage{dsfont}
\usepackage{amsfonts}
\usepackage{amssymb}
\usepackage{aliascnt}
\usepackage{commath}
\usepackage{mathrsfs}
\usepackage{enumerate}
\usepackage{comment}
\usepackage[bookmarks=false,pdfstartview=FitH, pdfborder={0 0 0}, colorlinks=true,citecolor=blue, linkcolor=blue]{hyperref}
\usepackage[numbers, sort&compress]{natbib}
\newtheorem{theorem}{Theorem}[section]
\newaliascnt{conj}{theorem}
\newaliascnt{cor}{theorem}
\newaliascnt{lemma}{theorem}
\newaliascnt{fact}{theorem}
\newaliascnt{claim}{theorem}
\newaliascnt{prop}{theorem}
\newaliascnt{definition}{theorem}

\newtheorem{cor}[cor]{Corollary}
\newtheorem{lemma}[lemma]{Lemma}

\newtheorem{prop}[prop]{Proposition}

\newtheorem{definition}[definition]{Definition}

\aliascntresetthe{conj} \aliascntresetthe{cor}
\aliascntresetthe{lemma} \aliascntresetthe{fact}
\aliascntresetthe{claim} \aliascntresetthe{prop}
\aliascntresetthe{definition}

\theoremstyle{definition}
\newaliascnt{example}{theorem}

\aliascntresetthe{example}

\theoremstyle{remark}
\newaliascnt{rmk}{theorem}
\newtheorem{remark}[rmk]{Remark}
\aliascntresetthe{rmk} 

\def\sek~{\S{}}

\numberwithin{equation}{section}
\newcommand{\curv}{\operatorname{curv}}
\newcommand{\vol}{\operatorname{\mu}}
\newcommand{\rvol}{\operatorname{vol}}

\newcommand{\dis}{\operatorname{d}}

\newcommand{\BA}{\operatorname{BA}}

\newcommand{\C}{\operatorname{Cone}}

\newcommand{\alex}{\operatorname{Alex}}
\newcommand{\RR}{\mathds{R}}

\newcommand{\NN}{\mathds{N}}
\renewcommand{\SS}{\mathbf{S}}

\begin{document}
\title{Volume estimates for  Alexandrov Spaces with convex boundaries}
\author{Jian Ge}
\address[Ge]{Beijing International Center for Mathematical Research, Peking University, Beijing 100871, P. R. China.}
\email{jge@math.pku.edu.cn}


\begin{abstract}
In this note, we estimate the upper bound of volume of closed positively or nonnegatively curved Alexandrov space $X$ with strictly convex boundary. We also discuss the equality case. In particular, the Boundary Conjecture holds when the volume upper bound is achieved. Our theorem also can be applied to Riemannian manifolds with non-smooth boundary, which generalizes Heintze and Karcher's classical volume comparison theorem. Our main tool is the gradient flow of semi-concave functions.
\end{abstract}

\maketitle
\section{Introduction}
Let $(M^{n}, g)$ be a smooth Riemannian manifold and $N^{m}\subset M$ be a submanifold. In \cite{HK1978}, Heintze and Karcher developed a general comparison theorem for N-Jacobi fields normal to the submanifold $N$. As an application, they were able to estimate the volume of a tube around the submanifold $N$ in terms of the lower curvature bound of $M$ and upper bound of the norm of mean curvature vectors of $N$ etc. But the technique does not apply if the embedding $N\to M$ is not smooth. For example the boundary of the Alexandrov Lens $\SS^{n-2}*[0,\alpha]$ defined in \cite{GMP2018}, with its intrinsic length distance, is a perfectly round sphere $\SS^{n-1}$, but the embedding is non-smooth. To address this kind of problem, we meet non-smooth spaces inevitably. Alexandrov spaces can be viewed as a generalization of Riemannian manifolds under lower sectional curvature bound, but with some singularities. Since the spaces that we are interested in are non-smooth in nature, we cannot expect to estimate the whole tube around $N$ without addressing the geometry of the singularities of the embedding $N\to M$. Nevertheless, we are able to estimate the one-side neighborhood of a convex hypersurface. We use \cite{BBI2001}, \cite{BGP1992} and \cite{Pet2007} as our references for Alexandrov spaces.

We use $\alex^{n}(\kappa)$ to denote the set of all $n$-dimensional Alexandrov spaces with lower curvature bound: $\curv \ge \kappa$. For $X\in \alex^{n}(\kappa)$ with $\partial X\ne \varnothing$, the boundary $\partial X$ is geodesically convex by the very definition of Alexandrov spaces. We need the following definition from Alexander and Bishop to measure the convexity of $\partial X \subset X$ quantitatively at footpoints. Recall that $p\in \partial X$ is called a \emph{footpoint}, if $p$ is the endpoint of a shortest geodesic from a point $x\in X\setminus \partial X$ to $p$ such that $\dis(x, \partial X)=\dis(x, p)$. Such shortest geodesic is necessarily unique, and the space of directions at $p$ is the spherical cone over $\partial \Sigma_{p}(X)$, cf. \cite{Per1991}. Recall that a \emph{chord} of $\partial X$ is a shortest geodesic in $X$ connecting two points in $\partial X$.

\begin{definition}[\cite{MR2595678}, Def.4.1]\label{def:BA}
Let $X$ be an $n$-dimensional Alexandrov space with non-empty boundary $\partial X$. For $p\in \partial X$, the \it{base angle at} $p$ of a chord $\gamma$ of $\partial X$ is the angle formed by the direction of $\gamma$ and $\partial(\Sigma_p(X))$, where $\Sigma_{p}(X)$ is the space of directions at $p$ of $X$. We call the boundary $\partial X$ has extrinsic curvature $\ge A$  in the base-angle sense at $p$ or $\BA(p, \partial X)\ge A$, if the base angle $\alpha$ at $p$ of chord of length $r$ from $p$ satisfies
$$
\liminf_{r\to 0}\frac{2\alpha}{r}\ge A.
$$
\end{definition}
It can be verified that in the Riemannian setting, this is equivalent to a positive lower bound on the principle curvatures of a smooth hypersurface. Note that for $A=0$, the $A$-convexity is simply convexity, which is automatically satisfied for any Alexandrov space with boundary.

We hereafter call the boundary $\partial X$ is $A$-convex, if the $\BA(p, \partial X)\ge A$ at each foot point $p$, which will be written as $\BA(X, \partial X)\ge A$.

Since we are interested in the volume estimates, the boundary $\partial X$ is assumed to be compact, even though many of our estimates hold locally. We always assume $\partial X$ is equipped with the length metric induced from $X$. It is still a conjecture that $\partial X$ itself is an Alexandrov space with the same lower curvature bound as of that of $X$. See the recent development of this conjecture in \cite{GMP2018}. 

To state our theorem, let's recall for an $n$-dimensional Alexandrov space $X$, the $n$-dimensional Hausdorff measure is denoted by $\vol_{n}(X)$, which will be called the \emph{volume} of $X$. The $n$-dimensional \emph{rough volume} is defined by 
$$
\rvol_{n}(X)=\lim_{\varepsilon\to 0}{\varepsilon}^{n}\beta_{X}(\varepsilon),
$$
where $\beta_{X}(\varepsilon)$ is the maximal cardinality of an $\varepsilon$-separated subset of $X$, that is, a set of points $\{x_{i}\}$ such that $\abs{x_{i}x_{j}}\ge \varepsilon$ for $i\ne j$. It is proved by N. Li in his thesis \cite{Li2010} that 
$$
\vol_{n}(X)=C_{n}\rvol_{n}(X)
$$
for some constant $C_{n}$ depends only on the dimension. We will use Hausdorff measure in this note.

Our first result is the following volume estimates:
\begin{theorem}\label{thm:-1}
Let $X\in \alex^{n}(0)$ be an Alexandrov space with 1-convex boundary $\partial X$. For $0< r \le 1$ we have
$$
\vol_{n}(B(\partial X, r))\le \vol_{n-1}(\partial X)\int_{0}^{r}(1-t)^{n-1}dt,
$$
with equality holds if and only if $B(\partial X, r)$ is isometric to the warped product space $\partial X\times_{1-t}[0, r]$.
Let $Y\in \alex^{n}(1)$ be an Alexandrov space. For $0< r\le \pi/2$, we have
$$\vol_{n}(B(\partial Y, r))\le \vol_{n-1}(\partial Y)\int_{0}^{r}\cos^{n-1}(t)dt$$
with equality holds if and only if  $B(\partial Y, r)$ is isometric to the warped product space $\partial Y\times_{\cos(t)}[0, r]$.
\end{theorem}
Warped products shown in the rigidity part of \autoref{thm:-1} are in fact  part of the standard cone:
\begin{definition}\label{def:cone}
Let $V$ be a metric space . For $T\in (0, \infty)$, we denote the linear cone $[0,T]\times_t V$ by $\C_0^T(V)$. For $T\in (0, \pi]$, we denote the spherical cone $[0, T]\times_{\sin(t)} V$ by $\C_{1}^{T}(V)$
\end{definition}

Clearly by the construction of cones, the space of directions at the conic point is isometric to $V$. It follows that unless $V\in \alex^{n-1}(1)$, the cones constructed above are not Alexandrov space. Therefore, we make the following remark

\begin{remark}
It is an easy corollary that when equalities hold in \autoref{thm:-1}, the boundaries $\partial X$ and $\partial Y$ are both lie in $\alex^{n-1}(1)$. Namely the boundary conjecture holds in this case. In fact this implies the generlized boundary  conjecture which is proposed in \cite{GL2018a} holds in this case.
\end{remark}

In \cite{MR2595678},  the authors estimated the inradius, i.e. the maximal radius of ball that inscribed in the space, of the space $X$ and $Y$. Namely under the assumptions of \autoref{thm:-1}, Alexander and Bishop  showed that $inrad(X)\le 1$ and $inrad(Y)\le \pi/2$. Therefore we have 
\begin{cor}
Let $X$ and $Y$ be as in \autoref{thm:-1}. Then $\vol_{n}(X)\le \vol_{n}(\C_{0}^{1}(\partial X))$ and $\vol_{n}(Y)\le \vol_{n}(\C_{1}^{\pi/2}(\partial Y))$
\end{cor}

\begin{remark}
The classical Bishop-Gromov volume comparison theorem bounds the volume of $r$-ball in Riemannian manifold with lower Ricci curvature bound by the one in the simply connected space form. Such comparison result has been generalized to Alexandrov setting, cf \cite{BGP1992} as well as \cite{LR2012}.  Other metric properties of $X$ and $Y$ are studied in \cite{GL2018a}.
\end{remark}

In \cite{Gro1991}, Gromov interprets the meaning of sign of various curvatures using the evolution of equal-distance hyper-surfaces. Roughly speaking, the sign of the curvature is reflected by the convexity of the equal-distance evolution a hypersurface, such convexity is controlled by the Riccati equation or Hessian/Laplacian comparison in the smooth setting. In this note, we estimate quantitatively the convexity of the level let of $\rho$, which is reflected by the concavity of $\rho$, where
$$
\rho =\dis_{\partial X}(\cdot)=\dis(\partial X, \cdot).
$$
\begin{theorem}[Hessian Comparison]\label{prop:RCT}
Let $X$ and $Y$ be as in \autoref{thm:-1}. Let $G(t)=\rho^{-1}(t)$ be the level set of $\rho$ and $\Omega^{t}:=\rho^{-1}([t, a])$ be the level set and super level set of $\rho$ in $X$ or $Y$, then we have the base-angle lower bound:
\begin{enumerate}
\item $\BA(\Omega^{t}, G(t))\ge \frac{1}{1-t}$ in $X$,
\item $\BA(\Omega^{t}, G(t))\ge \tan(t)$ in $Y$.
\end{enumerate}
\end{theorem}

In Section 1, we recall the definition of gradient flow of semi-concave functions. Several important estimates of gradient flows will be proved in this section, which play an essential role in our proof of volume bounds. Section 2 is devoted to the proof of volume upper bounds. The Hessian comparison is proved in Section 3.

Acknowledgment: We would like to thank Professor Takashi Shioya for his interest in our work.

\section{Gradient Curve for semi-concave functions}
Our main reference for semiconcave functions is \cite{Pet2007}. First, let's recall the definition of gradient curves of $\rho$, which is semi-concave if $X\in\alex^{n}(\kappa)$. Let $f:\RR\to \RR$ be a smooth function, with 
\begin{equation}\label{eq:f}
f'>0,\ \  f''>0.
\end{equation}
Let $X\in \alex^{n}(\kappa)$, with $\kappa =1$ or $0$. Let $a> 0$ be the positive number such that $\rho(X)=[0, a]$. Denote the sub-level  and super-level set  $\rho$ by 
$$\Omega_{t}=\rho^{-1}[0, t], \Omega^{t}=\rho^{-1}[t, a], $$ denote the level set by 
$$G(t)=\partial \Omega_{t}=\rho^{-1}(t).$$ 
We set
$$
F(x):=f(\rho(x)).
$$
For any $p\in \partial X=G(0)$, denote by $\alpha_{p}$ the gradient curve of $F$ with $\alpha_{p}(0)=p$, that is:
$$
\alpha_{p}^{+}(t)=\nabla_{\alpha_{p}(t)}F,
$$
where $\alpha^+(t)\in T_{\alpha(t)}X$ is the right tangent vector to $\alpha$ at time $t$. For any $\bar{w}\in G(b)$ with $b<a$, let $w\in \partial X$ be a foot point of $\bar{w}$. Clearly,  the gradient curve $\alpha_{w}$ coincide with the geodesic from $w$ to $\bar{w}$, possibly with a different parametrization. 

For a fixed $b<a$, and any $p\in\partial X$, we let $T_{p}=T_{p}(b)\in \RR$ defined by $\alpha_{p}(T_{p})\in G(b)$. Since $f$ is monotonic increasing, such $T_{p}$ exists and necessarily unique. We have the following lemma characterizes geodesics as the fastest gradient curves start from $G(0)$ to $G(b)$.
\begin{lemma}\label{lem:time}
Let $\bar w\in G(b)$ and $w\in G(0)$ be a footpoint of $\bar w$. Then for any $p\in G(0)$, we have
$$
T_{w}\le T_{p},
$$
with equality holds if and only if $\alpha_{p}$ coincides with the shortest geodesic segment from $\alpha_{p}(T_{p})$ to $G(0)$.
\end{lemma}
\begin{proof}
For any $p\in G(0)$, define $h_{p}(t)=\rho(\alpha_{p}(t))$. Clearly $h_{p}(0)=0$. One calculate the derivative of $h_{p}$ as
$$
\frac{d h_{p}}{dt}=f'(h_{p})(|\nabla\rho|^{2})|_{\alpha_{p}(t)}.
$$
Since $\rho$ is semi-concave, $|\nabla \rho|$ is well defined. Clearly $|\nabla\rho|\le 1$ everywhere and $(|\nabla\rho|)|_{\alpha_{w}(t)}\equiv 1$, for $t\in [0, T_{w}]$. Therefore, we have
$$
T_{w}=\int_{0}^{b}\frac{dh_{w}}{f'(h_{w})}\le \int_{0}^{b}\frac{dh_{p}}{f'(h_{p})|\nabla \rho|^{2}}=T_{p}.
$$
with equality if and only if $(|\nabla\rho|)|_{\alpha_{p}(t)} \equiv 1$, i.e. the conclusion of lemma holds.
\end{proof}

Next, we study the concavity of $F=f(\rho)$. Essentially, we only interested in the concavity of $\rho$, which reflects the geometry of the space $X$. However, due to the linearity of $\rho$ along geodesic segments that realize the distance between point $x$ and the boundary $\partial X$, we usually precompose a convex increasing function $f$ to get the estimate of concavity of $F$ for all directions. Although any function $f$ satisfies \eqref{eq:f} will work for our estimates, however, the following choice of $f$ is rather typical:
\begin{prop}\label{prop:AB}
Let $X$ be a compact $n$-dimensional Alexandrov space with nonempty boundary $\partial X$. 
\begin{enumerate}
\item (Theorem 1.8, \cite{MR2595678}) If $\curv(x)\ge 0$ and $\partial X$ is $1$-convex in the Base-Angle sense. Then the function $f(\rho(x)):=-(1-\rho(x))^{2}/2$ is $-1$-concave.
\item (Theorem 3.3.1, \cite{Pet2007}) If $\curv(X)\ge 1$, then $f(\rho(x)):=\sin(\rho(x))$ is $-f(\rho)$-concave.
\end{enumerate}
\end{prop}

For a continuous function $\phi:[0, a]\to \RR^{+}$, we call $F$ is a $\phi(\rho)$-concave function or the \emph{modulus of concavity} function of $F$ is $\phi$, if for any $t\in [0, a)$, $F$ is $\phi(t)$-concave in the interior of $\Omega_{t}$. One can easily see $\phi$ has to be a \emph{nonincreasing} function for $t\in [0, a)$. This is a very special type of concavity, since the concavity depends only on the level-sets of $\rho$, which makes the estimates on the gradient curves of different parametrizations possible. Using our language, \autoref{prop:AB} implies the modulus of concavity for $F$ is $\phi(t)=-1$ for $\kappa =0$ and $\phi(t)=-\sin(t)$ for $\kappa=1$.

A closed related concept, called the Sharafutdinov flow $\Psi^t: X\to X$ is defined by Perelman in \cite{Per1991}, by generalizing Sharafutdinov's original construction for Riemannian manifolds. It can be constructed as follows: Let $p\in \partial X$ then the flow curve $\tilde\alpha(t)=\Psi^t(p)$ satisfies the differential equality
$$
\tilde\alpha^+(t)=\frac{\nabla_{\tilde\alpha(t)}\rho}{|\nabla_{\tilde\alpha(t)}\rho|^2},
$$
with initial condition $\tilde\alpha^{+}(0)=p$, where $\tilde\alpha^+(t)\in T_{\tilde\alpha(t)}X$ is the right tangent vector to $\tilde\alpha$ at time $t$. Clearly flow curves of $\Psi$ are just gradient curves of $\rho$ with a different parametrization, such that it flows level set to level set, i.e. $$\Psi^t(G(s))=G(s+t),$$ for $t+s\le a$. With the help of $\Psi$, Perelman then defined the so called Sharafutdinov retraction which retracts the whole $X$ onto its soul. We summarize the properties of the retraction map: $sh$ as follows
\begin{prop}[\cite{Per1991}]\label{prop:PerlSh}
Let $X$ be as above. Then for any $t\le a$, there is a unique surjective map, $sh^t$, called Sharafutdinov retraction
$$sh^t: X\to \Omega_{t},$$ 
such that $sh^t|_{\Omega_{t}}=id$ and it is a short map, i.e. $sh^t$ is distance non-increasing. 
\end{prop}

Given $p, q\in G(0)$, we want to estimate $\dis(\Psi^{t}(p), \Psi^{t}(q))$. Since the Sharafutdinov flow curve coincide with the gradient curve of $\rho$ as well as the ones of $f(\rho)$, who's concavity is well understood by \autoref{prop:AB}, we estimates the distance between $\alpha_{p}$ and $\alpha_{q}$ instead. Such an estimate is rather easy for $\kappa =0$ case, since by \autoref{prop:AB}, $f(\rho)$ is $-1$-concave. i.e. the concavity is uniform on all $X$. The situation is a little complicated for $\kappa=1$. But the special type of concavity of $F$, i.e. $\phi$-concavity, makes the estimates possible for non-uniform concavity and non-uniform time gradient flows. We first recall the following estimates for gradient curves of $\lambda$-concave function, for uniform $\lambda\in \RR$ and uniform time $t$:

\begin{lemma}[Lemma 2.1.4 \cite{Pet2007}]\label{lem:pet}
Let $A$ be an Alexandrov space. Let $F:A\to \RR$ be a $\lambda$-concave function and $\alpha, \beta: [0, \infty)\to A$ be two $F$-gradient curves with $\alpha(0)=p, \beta(0)=q$. Then for $t>0$
$$
\dis(\alpha(t), \beta(t))\le e^{\lambda t}\dis(p, q).
$$
\end{lemma}

The following two theorems are essential for our estimates of the volume upper bound.
\begin{theorem}[Gradient curves in non-negatively curved spaces]\label{thm:gcurve:0}
For $X\in \alex^{n}(0)$ with $BA(\partial X, X)\ge 1$. Let $T\in (0, a)$, then the Sharafutdinov flow satisfies the following estimate 
$$\dis(\Psi^{T}(p), \Psi^{T}(q))\le (1-T) \dis(p, q), $$ 
for any $p, q\in G(0)$.
\end{theorem}
\begin{proof}
By \autoref{prop:AB}, $F=f(\rho)$ is $-1$-concave. As above, let $\alpha_{p}$ be the gradient curve of $F=f(\rho)$ starting from $p\in G(0)$. Let $t_{p}, t_{q}\in \RR$ be the times such that
$$
\bar{p}:=\alpha_{p}(t_{p})=\Psi^{T}(p)\in G(T),\ \ \bar{q}:=\alpha_{q}(t_{q})=\Psi^{T}(q)\in G(T).
$$
Therefore, we need to estimate the flow times $t_{p}, t_{q}$ and then estimate $\dis(\alpha_{p}(t_{p}),\alpha_{q}(t_{q}))$. 

Let us firstly consider a special case: $p$ (resp. $q$) is a footpoint of $\bar{p}$(resp. $\bar{q}$). Since $\rho$ and $F$ share the same gradient curves(as a set), the $\alpha_{p}$ and $\alpha_{q}$ are nothing but re-parametrized geodesic from $\bar{p}$ to $p$ and $\bar{q}$ to $q$ with length equals to $T$. Therefore $|\nabla\rho|_{\alpha_{p}(t)}=1$ for $0\le t\le t_{p}$ and $|\nabla\rho|_{\alpha_{q}(t)}=1$ for $0\le s\le t_{q}$. In such a case, let $h_{p}(t)=\rho(\alpha(t))$, we have
$$
h_{p}'(t)=f'(h_{p}(t))(|\nabla\rho|_{\alpha_{p}(t)})^{2}=1-h_{p}(t).
$$
Combining with the initial condition $h_{p}(0)=0$, we get:
$$
t_{p}=-\ln(1-T).
$$
Similarly, we have $t_{q}=-\ln(1-T)$. Therefore, if both $p$ and $q$ are footpoints, it follows from \autoref{lem:pet} that:
\begin{align*}
\dis(\bar p, \bar q)&=\dis(\Phi_F^{-\ln(1-t_0)}(p), \Phi_F^{-\ln(1-t_0)}(p))\\
&\le e^{\ln(1-T)}\dis(p, q)\\
&\le (1-T)\dis(p, q).
\end{align*}
For other cases, for example $p$ is not a footpoint of $\bar p$, then by \autoref{lem:time}, we have $t_{p}>-\ln(1-T)$. Then let $\tilde p=\Phi_f^{-\ln(1-T)}(p)$ and  $\tilde q=\Phi_f^{-\ln(1-T)}(q)$. Therefore by \autoref{lem:pet}, we still have
$$
\dis(\tilde p, \tilde q)\le (1-T) \dis(p, q).
$$
On the other hand the Sharafutdinov retraction $sh^{T}$ retracts $X$ onto $\Omega(T)$ along the gradient curve of $\rho$. Therefore $sh^{T}(\tilde p)=\bar p$ and $sh^{T}(\tilde q)=\bar q$. Hence by \autoref{prop:PerlSh}, we have
$$
\dis(\bar p, \bar q)\le \dis(\tilde p, \tilde q) \le (1-T)\dis(p, q).
$$
This finishes the proof.
\end{proof}

\begin{theorem}[Gradient curves in positively curved spaces]\label{thm:gcurve:1}
For $X\in \alex^{n}(1)$, and $T\in (0, a)$. Then the Sharafutdinov flow satisfies the following estimate $$\dis(\Psi^{T}(p), \Psi^{T}(q))\le \cos(T) \dis(p, q), $$ for any $p, q\in G(0)$.
\end{theorem}
\begin{proof}
For any $N\in \NN$, we set $x_{i}:=iT/N$ for $i=0, \cdots, N$. For any $p, q\in G(0)$, let $\alpha_{p}$ and $\alpha_{q}$ be the gradient of $F=f(\rho)$-gradient curve as before, where $f=\sin(x)$. Pick any point $\bar{w}\in G(T)$ and let $w\in G(0)$ be a footpoint of $\bar{w}$. Let $\alpha_{w}$ be the gradient flow of $F$ starting from $w$ to $\bar{w}$. We will use the gradient curve $\alpha_{w}$ as an auxiliary tool in order to keep tracking the flow times of the gradient curves $\alpha_{p}$ and $\alpha_{q}$. 

By \autoref{prop:AB}, $F$ is $-\sin(\rho)$-concave in $\Omega_{0}$. Moreover, $F$ is $-\sin(x_{i-1})$-concave in $\Omega_{x_{i-1}}$, for $i=1, \cdots, N$. Let $\{t_{p,i}\}, \{t_{q,i}\}$ and $\{s_{i}\}$ be the set of numbers such that
$$
p_{i}:=\alpha_{p}(t_{p, i})\in G(x_{i}),\ \ q_{i}:=\alpha_{q}(t_{q, i})\in G(x_{i}),\ \ w_{i}:=\alpha_{w}(s_{i})\in G(x_{i}).
$$
Let $$\Delta_{i-1}=\min\{(t_{p,i}-t_{p, i-1}), (t_{q, i}-t_{q, i-1})\}.$$
By \autoref{lem:time},
\begin{equation}\label{eq:1:good}
s_{i}-s_{i-1}\le \Delta_{i-1}.
\end{equation}
Since $F$ is $-\sin(x_{i-1})$-concave in $\Omega^{x_{i-1}}$, we can apply \autoref{lem:pet} to the flow of $F$ for time $\Delta_{i-1}$, in the domain $\Omega_{x_{i}}\setminus \Omega_{x_{i-1}}$. i.e. for $i=1, \cdots, N$, we have the following estimate:
\begin{align*}
\dis(p_{i}, q_{i})&\le \dis(\Phi_{F}^{\Delta_{i-1}}(p_{i-1}), \Phi_{F}^{\Delta_{i-1}}(q_{i-1}))\\
&\le e^{-\sin(x_{i-1})\Delta_{i-1}}\dis(p_{i-1}, q_{i-1})\\
&\le e^{-\sin(x_{i-1})(s_{i}-s_{i-1})}\dis(p_{i-1}, q_{i-1})
\end{align*}
where we apply \autoref{prop:PerlSh} to get the first inequality. Last inequality is due to $-\sin(x)|_{[0, \pi/2]}\le 0$ and \eqref{eq:1:good}. Summing up for $i=1, \cdots, N$, we have
\begin{equation}\label{eq:1:dis}
\dis(p_{N}, q_{N})\le e^{-\sum_{i=1}^{N}\sin(x_{i})(s_{i}-s_{i-1})}\dis(p, q).
\end{equation}
Let $h(s)=\rho(\alpha_{w}(s))$. Since $\alpha_{w}$ coincide with the shortest geodesic from $w$ to $\bar{w}$, we know $|\nabla \rho|_{\alpha_{w}(s)}=1$. Therefore
\begin{equation}\label{eq:cos}
h'(s)=f'(h(s))=\cos(h(s)).
\end{equation}
Letting $N\to \infty$:
\begin{align*}
\lim_{N\to \infty}-\sum_{i=1}^{N}\sin(x_{i-1})(s_{i}-s_{i-1})&=\lim_{N\to \infty}-\sum_{i=1}^{N}\sin(h(s_{i-1}))(s_{i}-s_{i-1})\\
&=-\int_{0}^{s_{N}}\sin(h(s))ds\\
&=-\int_{0}^{T}\frac{\sin(h)}{\cos(h)}dh\\
&=\ln(\cos(T)).
\end{align*}
where we use $dh=\cos(h)ds$ derived from \eqref{eq:cos}. Therefore, talking limit as $N\to \infty$ in \eqref{eq:1:dis}, we have
$$
\dis(\Psi^{T}(p), \Psi^{T}(q))\le \cos(T) \dis(p, q).
$$
This finishes the proof.
\end{proof}

\begin{remark}
Our estimates of the gradient curves above can be proved for $X\in \alex^{n}(\kappa)$ with $A$-convex boundary, for general $\kappa \ge 0$ and $A>0$, either by rescaling or by invoking \autoref{prop:RCT}.
\end{remark}

\section{Proof of Volume Upper Bounds}

By the co-area formula applied to the distance function $\rho$, we have
$$
\vol_{n}(X)=\int_0^a A(t) dt,
$$
where $A(t)=\vol_{n-1}(G(t))$, the $(n-1)$-dimensional Hausdorff measure of the level set $G(t):=\rho^{-1}(t)$. In the model space $\C_0^1(\partial X)$ (resp. $\C_{1}^{\pi/2}$) we denote the level set of the distance function $\rho=\dis(\partial X ,\cdot)$ in $\C_0^1(\partial X)$ (resp. $\C_{1}^{\pi/2}$) by $G^*(t)$. We set $A^*(t)=\vol_{n-1}(G^*(t))$. Since $a\le 1$ (resp. $a\le \pi/2$) and 
$$
\vol_{n}(\C_0^1(\partial X))=\int_0^1 A^*(t)dt,
$$
resp.
$$
\vol_{n}(\C_1^{\pi/2}(\partial X))=\int_0^{\pi/2} A^*(t)dt,
$$
therefore to prove the volume estimates from above,  it suffices to show 
\begin{equation}\label{eq:area}
A(t)\le A^*(t).
\end{equation}
The following proposition is clear from the definition of Hausdorff measure:
\begin{prop}\label{prop:haus}
Let $X$ and $Y$ be two metric spaces and $f:X\to Y$ be a Lipschitz map with Lipschitz constant $L$, then
$$\mu_{d}(f(X))\le L^{d}\mu_{d}(X).$$
\end{prop}

Cleary $A(0)=A^*(0)$ and $A^*(t)=(1-t)^{n-1}A^*(0)$ (resp. $A^*(t)=\cos^{n-1}(t)A^*(0)$). By \autoref{thm:gcurve:0} , \autoref{thm:gcurve:1} and \autoref{prop:haus} we have the following, which proves \autoref{thm:-1}.
\begin{cor}
For $0\le T\le a$, the $(n-1)$-dimensional Hausdorff measure of $G(T)$ satisfies
$$
A(T)=\mu_{n-1}(G(T))\le (1-T)^{n-1}A(0)=A^*(T),\ \  {\rm if}\ \ \kappa=0;
$$
$$
A(T)=\mu_{n-1}(G(T))\le \cos^{n-1}(T)A(0)=A^*(T),\ \  {\rm if}\ \ \kappa=1;
$$
\end{cor}

\begin{proof}[Discussion of equalities in \autoref{thm:-1}]
If the equality of the volume estimate in $X$ holds in \autoref{thm:-1} for some $r>0$, then the time estimate in \autoref{lem:time} implies that every point $p\in \partial X$ is a foot point for some $\bar p\in G(r)$. Therefore, every point in $\Omega_{r}$ lies in some geodesic from $G(0)$ to $G(r)$. Moreover the equality in \autoref{thm:gcurve:0}  holds, which implies that for every pair of geodesics $p\bar p$ and $q\bar q$ from $\partial X$ to $G(r)$ span a totally geodesic flat surface in $X$, which gives the desired warped product structure on $\Omega_{r}=B(\partial X, r)$. The discussion on $Y$ is similar.
\end{proof}

\section{Estimate on the convexity of the level sets}
In this section, we estimate the convexity of the level set of $\rho$, which can be viewed as Hessian comparison theorem in the smooth setting. 
\begin{proof}[Proof of \autoref{prop:RCT}]
For any point $p\in G(t)$, we want to estimate the convexity of $G(t)$ in the Base-Angle sense. First, by the definition, we can assume $p$ is a foot point, hence by the first variation:
$$
\Sigma_{p}(X)=\SS(\partial \Sigma_{p}(\Omega^{t})),
$$
where $\SS(Y)$ denotes the spherical suspension of the space $Y$. The suspension points $N, S\in \Sigma_{p}(X)$ can be viewed as follows: we can identify $N$ with the inner pointing normal vector and identify $S$ as the $\uparrow_{p}^{q}$, where $q\in \partial X$ such that $\dis(p, q)=\rho(p)$. For any $s>0$ small, let $x\in G(t)$ with $\dis(p, x)=s$. We want to estimate 
$$\theta=\theta(s):=\dis(\Uparrow_{p}^{x}, \partial \Sigma_{p}(\Omega^{t})).$$ Clearly $\rho(x)=\rho(p)=t$. Let $\xi\in \partial \Sigma_{p}(\Omega^{t})$ such that
$$
\dis(\Uparrow_{p}^{x}, \partial \Sigma_{p}(\Omega^{t}))=\dis(\Uparrow_{p}^{x}, \xi).
$$
Let $\gamma_{\xi}$ be a quasi-geodesic in $X$ with initial velocity $\xi$. Then by the triangle comparison theorem applied to the triangle $\overline{px}$, $\overline{x\gamma_{\xi}(s)}$, $\gamma_{\xi}[0, s]$, we have
\begin{equation}\label{eq:ricaati00}
l(s):=\dis(x, \gamma_{\xi}(s))\le \sqrt{2s^{2}-2s^{2}\cos(\theta)}.
\end{equation}
Note that, we use the Euclidean comparison triangle for both cases in order to simply the calculation. One could have used the spherical comparison for $\curv\ge 1$, but it does not gain any extra information.
Let $M_{i}(t)$ be the modulus of concavity of the function $\rho$ at the level set $G(t)$, for the case $\curv(X)\ge i$, $i=0, 1$. An easy calculation from \autoref{prop:AB} gives
$$
M_{0}=\frac{-1}{1-t},\ \ \ M_{1}=-\tan(t).
$$
Set $h(s)=\rho(\gamma_{\xi}(s))$. Thus $h(0)=t$, $h'(0)=0$ and $h''(0)\le M_{i}(t), i=1,2.$
Therefore, for $s$ small, we have:
\begin{equation}\label{eq:ricaati01}
h(s)\le t+\frac{M_{i}}{2}s^{2}+o(s^{2}), \ \ \ i=1,2.
\end{equation}
Combining \eqref{eq:ricaati00} and \eqref{eq:ricaati01}, we can estimate $\dis(x, \partial X)$
\begin{eqnarray*}
t=\rho(x)&\le&l(s)+h(s)\\
&\le& \sqrt{2s^{2}-2s^{2}\cos(\theta)}+t+\frac{M_{i}}{2}s^{2}+o(s^{2}).
\end{eqnarray*}
that is
\begin{eqnarray*}
\theta&\ge&\arccos\left(1-\left(\frac{1}{8}M_{i}^{2}s^{2}+o(s^{2})\right)\right)\\
&=&\sqrt{2}\sqrt{\frac{1}{8}M_{i}^{2}s^{2}+o(s^{2})}+o(s)\\
&=&\frac{s}{2}M_{i}+o(s),
\end{eqnarray*}
where in the first equality we use the Taylor expansion of $\arccos(1-x)$ in terms of $\sqrt{x}$. Therefore, the conclusion follows by the definition of Base-Angle.
\end{proof}

\bibliographystyle{alpha}
\bibliography{mybib}
\end{document}